\documentclass[10pt]{amsart}
\usepackage{a4}
\usepackage{amssymb}
\usepackage{array}
\usepackage{booktabs}
\usepackage{multirow}
\usepackage{hhline}
\usepackage{color}

\newtheorem{theorem}{Theorem}

\newtheorem{question}{Question}
\newtheorem{corollary}[theorem]{Corollary}
\newtheorem{lemma}[theorem]{Lemma}
\newtheorem{proposition}[theorem]{Proposition}
\newtheorem{definition}[theorem]{Definition}

\newtheorem{remark}[theorem]{Remark}

\numberwithin{theorem}{section}
\numberwithin{equation}{section}
\newcommand{\A}{\mathcal{A}}

\newcommand{\Comp}{\mathbb{C}}

\newcommand{\g}{\mathbb{G}}

\newcommand{\n}{\mathbb{N}}

\newcommand{\R}{\mathbb{R}}

\newcommand{\z}{\mathbb{Z}}

\begin{document}
\title{On the Similarity Problem for Locally Compact Quantum Groups}

\author{Michael Brannan}
\author{Sang-Gyun Youn}

\address{Michael Brannan:  Department of Mathematics, Mailstop 3368, Texas A$\&$M University,
 College Station, TX 77843-3368, USA}
\email{mbrannan@math.tamu.edu}

\address{Sang-gyun Youn: Department of Mathematical Sciences, Seoul National University,
San56-1 Shinrim-dong Kwanak-gu, Seoul 151-747, Republic of Korea}
\email{yun87654@snu.ac.kr}


\begin{abstract}
A well-known theorem of Day and Dixmier states that any uniformly bounded representation of an amenable locally compact group $G$ on a Hilbert space is similar to a unitary representation. Within the category of locally compact quantum groups, the conjectural analogue of the Day-Dixmier theorem is that every completely bounded Hilbert space representation of the convolution algebra of an amenable locally compact quantum group should be similar to a $\ast$-representation.  We prove that this conjecture is false for a large class of non-Kac type compact quantum groups, including all $q$-deformations of compact simply connected semisimple Lie groups. On the other hand, within the Kac framework, we prove that the Day-Dixmier theorem does indeed hold for several new classes of examples, including amenable discrete quantum groups of Kac-type.

\end{abstract}

\maketitle

\section{Introduction}

For a locally compact group $G$, the question of the  unitarizability of uniformly bounded representation has  quite a long history. The begining of this story started with the following general result of Day and Dixmier, extending previous work of Sz.-Nagy \cite{Na47} on the particular case $G=\z$. 

\begin{theorem}[Day-Dixmier Theorem \cite{Da50},\cite{Di50}]
If a locally compact group $G$ is amenable, then every uniformly bounded Hilbert space representation $\pi:G\rightarrow B(H_{\pi})$ admits an invertible $T\in B(H_{\pi})$ such that $T\circ \pi(\cdot) \circ T^{-1}$ is a unitary representation.
\end{theorem}

Since there is a bijective correspondence between uniformly bounded representations $\pi:G\rightarrow B(H_{\pi})$ and bounded representations $\pi:L^1(G) \to B(H_\pi)$ (of the associated convolution algebra $L^1(G)$), the above celebrated work can be concisely described in terms of the so-called {\it similarity property} for $L^1(G)$. More precisely, we have that
\begin{enumerate}
\item every contractive representation $\pi:L^1(G) \to B(H_\pi)$ is a $\ast$-representation, and
\item  under the assumption of amenability of $G$, every bounded representation  $\pi:L^1(G) \to B(H_\pi)$  is similar to a $\ast$-representation.
\end{enumerate}

The question of whether the converse to the Day-Dixmier theorem holds is called Dixmier's problem and it is still open (although there are some notable partial results \cite{EhMa55},\cite{Pi07},\cite{EpMo09} and \cite{MoOz10}). A remarkable partial answer to Dixmier's problem was obtained by G. Pisier \cite{Pi98} for discrete groups and N. Spronk \cite{Sp02} for the general case by requiring a norm condition $\left \|T\right\| \left \|T^{-1}\right \|\leq \left \|\pi\right\|^2$. In other words, amenability of $G$ is equivalet to $L^1(G)$ having the the similarity property with {\it (completely bounded) similarity degree} $d_{cb}(L^1(G))\leq 2$. For more details, see Subsection \ref{pre:degree}.

Within the framework of locally compact quantum groups, it is natural to ask whether such known results generalize.  More precisely, let $\g=(L^{\infty}(\g),\Delta,\varphi,\psi)$ be a locally compact quantum group and let $L^1(\g) = L^\infty(\g)_*$ denote the associated convolution algebra.

\begin{question}\label{Q1}
Is every (completely) contractive representation $\pi:L^1(\g) \to B(H_\pi)$ automatically a $*$-representation?
\end{question}
\begin{question}\label{Q2}
Is every (completely) bounded representation $\pi:L^1(\g) \to B(H_\pi)$  similar to a $\ast$-representation, at least when $\g$ is amenable?
\end{question}

In the above questions, we impose the condition that our representations are completely bounded maps.  This is natural when working with genuine quantum groups, since for ordinary groups $G$, all bounded representations $\pi:L^1(G) = \text{MAX}(L^1(G)) \to B(H_\pi)$ are automatically completely bounded.  Moreover, in the quantum  case, any representation of $L^1(\g)$ that is  similar to a $\ast$-representation is automatically completely bounded.   We also note that the general assumption of complete boundedness on a representation $\pi:L^1(\g) \to B(H_\pi)$ is not redundant: \cite{ChSa13} established the existence of examples of bounded
 $\pi:L^1(\g) \to B(H_\pi)$ which are not completely bounded.  See also \cite{BrDaSa13}.   This leads us to the following definition.  

\begin{definition}
Let $\g$ be a locally compact quantum group.  We say that $\g$ (or $L^1(\g)$) has {\it the Day-Dixmier property} if the answers to Questions \eqref{Q1}-\eqref{Q2} are affirmative.
\end{definition}

The first investigation into the Day-Dixmier property for quantum groups was in \cite{BrSa10}. Here, the authors  considered the Fourier algebra $A(G)$ of a locally compact group $G$, which corresponds to the convolution algebra of the co-commutative dual quantum group $\widehat G$ (which turns out to always be amenable).  There, they showed that the Day-Dixmier property on $A(G)$ holds for all SIN(=small invariant neighborhood) groups.  They also observed more generally that for any locally compact group $G$, and any completely bounded representation $\pi:A(G) \to B(H_\pi)$, Question \ref{Q2} has an affirmative answer if and only if a certain related map $\check{\pi}$ is completely bounded. Here, $\check{\pi}$ is the (anti-)representation of $A(G)$ defined by $w\mapsto \pi(\check{w})$ where $\check{w}(g)=w(g^{-1})$.

For genuine locally compact quantum groups,  \cite{BrDaSa13} generalized the affirmative answer of Question \ref{Q2} on $A(G)$ to the case of amenable locally compact quantum groups. More precisely, they showed that any completely bounded representation $\pi:L^1(\g)\rightarrow B(H_{\pi})$ for which $\left \|\check{\pi}\right \|_{cb}<\infty$ is similar to a $\ast$-representation. Moreover, for compact quantum groups $\g$ of Kac type, the authors of \cite{BrDaSa13} showed that Day-Dixmier theorem holds in full generality without assumption on $\check{\pi}$. Here, $\check{\pi}$ is the (a priori unbounded) anti-representation of $L^1(\g)$ defined by $w\mapsto \pi((w^*)^{\sharp})$ where $\langle (w^*)^{\sharp},x\rangle=\langle w,S(x)\rangle$ and $S$ is the {\it antipode} map.  

In summary: the results of \cite{BrSa10, BrDaSa13} show that, with the exception of some small classes of amenable quantum groups (i.e., classical amenable groups, duals of SIN groups, compact Kac algebras, ...), establishing the Day-Dixmier property seems to require one to not only have complete boundedness of a given representation $\pi$, but also complete boundedness of the affiliated map $\check{\pi}$.  It is quite natural to ask whether the additional complete boundedness assumption on $\check{\pi}$ is in fact required.    Quite recently,  \cite{LeSaSp16} showed that the complete boudnedness of $\check{\pi}$ was indeed automatic for a large class of Fourier algebras.  More precisely, they tackled the similarity problem for $A(G)$ using tools more directly connected to Pisier's machinery \cite{Pi98}, proving that for a broad class of groups, $A(G)$ has the Day-Dixmier property with completely bounded similarity degree $d_{cb}(A(G)) \le 2$.  This work provides significant evidence to suggest that complete boundedness assumptions on $\check{\pi}$ are indeed unnecessary, at least for group duals $\widehat G$.  

Our first main objective in this paper is to show (by means of explicit examples) that the appearence of the anti-representation $\check{\pi}$ in the analysis of the Day-Dixmier property is indeed essential when working in the framework of general locally compact quantum groups.  More precisely, in Theorem \ref{thm:subexponential}, we show that any compact quantum group $\g$ with the Day-Dixmier property whose dual $\widehat \g$ has subexponential growth must be of Kac type.  This result, in particular, implies (cf. Corollary \ref{cor:Gq} ) that if $G$ is any compact simply connected semisimple Lie group, then its Drinfeld-Jimbo deformation $G_q$ ($0 < q < 1$) can never have the Day-Dixmier property.   Based on these results, we conjecture that every amenable locally compact quantum group with the Day-Dixmier property is automatically of Kac type. 

Our second objective in this paper is to establish some new classes of amenable Kac-type quantum groups which have the Day-Dixmier property. The examples include all of amenable discrete quantum groups of Kac-type and the duals of certain crossed products that are non-compact, non-discrete, non-commutative and non-cocommutative in general.

All of these new examples arise as consequences of Theorem \ref{thm:tracial-mean}, which follows along the same line of attack as the prior works \cite{BrSa10, BrDaSa13} where the above assumptions allow one to show that a given completely bounded representation $\pi:L^1(\g)\to B(H)$ automatically extends to a completely bounded homomorphism $\Phi$ from the enveloping C$^\ast$-algebra $C_0^u(\widehat \g)$ into $B(H)$ satisfying $\|\Phi\|_{cb} \le \|\pi\|_{cb}^2$.  Since by coamenability $C^u_0(\widehat \g)=C_0(\widehat{\g})$ is nuclear and nuclear C$^\ast$-algebras have completely bounded similarity degree 2, the fact that $d_{cb}(L^1(\g)) \le 4$ easily follows. 

One would hope for a better result in  Theorem \ref{thm:tracial-mean}, namely that $d_{cb}(L^1(\g)) \leq 2$.  We explain in Appendix \ref{appendix}, using different techniques more in line with \cite{Pi98, LeSaSp16}, how one can obtain $d_{cb}(L^1(\g)) \leq 2$ if $\g$ is a compact Kac algebra or an amenable discrete Kac algebra. We also note that $d_{cb}(L^1(\g))=1$ if and only if $L^{\infty}(\g)$ is finite dimensional in those cases.

The remainder of the paper is organized as follows: In Section \ref{section:prelim}, we introduce some of the basics of the theory of locally compact quantum groups and the completely bounded similarity degree that are needed for our work. Then we show in Section \ref{section:nonKac} that the Day-Dixmier property does not generally hold  within the category of compact quantum groups, and in Section \ref{section:positive} we establish the Day-Dixmier property for a class of examples with $\g$ is amenable and of Kac type with tracial left Haar weight.   Finally, in Appendix \ref{appendix}, we explain how to improve the similarity degree for some of the examples of Section \ref{section:positive}. 

\subsection*{Acknowledgements}   M. Brannan was supported by NSF Grant  DMS-1700267.  S. Youn was supported by a TJ Park Science Fellowship.   The authors would like to thank Hun Hee Lee for encouragement and comments. The authors are also grateful to Ebrahim Samei and Xiao Xiong who pointed out an error in an earlier version.

\section{Preliminaries} \label{section:prelim}

\subsection{Locally compact quantum groups}

We refer the reader to \cite{KuVa00, KuVa03, Va01} and the book \cite{Ti08} for an introduction to operator algebraic locally compact quantum groups.  Let us recall that a (von Neumann algebraic) 
{\it locally compact quantum group} is a von Neumann algebra $L^\infty(\g)$ equipped with a coassociative coproduct and left and right Haar weights.  The coproduct is a unital normal
$\ast$-homomorphism $\Delta : L^\infty (\g )\to L^\infty (\g )\overline{\otimes}L^\infty (\g )$
satisfying the coassociativity condition
\begin{align*}
(\Delta\otimes \mathrm{id} )\Delta = (\mathrm{id}\otimes\Delta )\Delta.
\end{align*}
The left and right Haar weights are normal semifinite faithful weights $\varphi$ and $\psi$
on $L^\infty (\g )$ such that for every $w\in L^\infty (\g )_*^+$ one has
\begin{align*}
\varphi ((w\otimes\mathrm{id} )\Delta (a)) = \varphi (a) w (1)
\end{align*}
for all $a\in L^\infty (\g )^+$ with $\varphi (a) < \infty$ and
\begin{align*}
\psi ((\mathrm{id}\otimes w )\Delta (a)) = \psi (a) w (1)
\end{align*}
for all $a\in L^\infty (\g )^+$ with $\psi (a) < \infty$.  
The predual of $L^\infty(\g)$ is written as $L^1(\g)$, and becomes a completely contractive 
Banach algebra with respect to the {\it convolution product} 
\[
w_1 \star w_2 = (w_1 \otimes w_2) \circ \Delta , \qquad w_1, w_2 \in L^1(\g).
\]

Associated to $\g$ is a canonical weakly dense sub-C$^*$-algebra of $L^\infty (\g )$,
written $C_0 (\g )$, which plays the role of the C$^*$-algebra of continuous functions vanishing 
at infinity in the case of ordinary groups.  Also, we denote by $M(\g)$ the dual space of $C_0(\g)$. The coproduct restricts to a unital $\ast $-homomorphism $\Delta: C_0(\g) \to M(C_0(\g) \otimes C_0(\g))$. 
The algebras $C_0 (\g )$ and $L^\infty (\g )$
are standardly represented on the GNS Hilbert space $L^2 (\g )$ associated to the left Haar weight.
In the case of a locally compact group, the notations $L^\infty (\g )$, $L^1(\g)$, $C_0 (\g )$, $L^2 (\g )$ and $M(\g)$ have their ordinary meaning.

There is a (left) {\it fundamental unitary operator} $W$ on $L^2(\g)\otimes L^2(\g)$ which
satisfies the {\it pentagonal relation} $W_{12} W_{13} W_{23} = W_{23} W_{12}$
and unitarily implements the coproduct $\Delta$ on $L^\infty(\g)$ via the formula 
$\Delta(x) = W^*(1\otimes x)W$.  Using $W$ one has 
$C_0(\g) =  \overline{\{(\mathrm{id} \otimes w)W: w \in B(L^2(\g))_*\}}^{\|\cdot\|}$, 
and one can define the {\it antipode} of $\g$ as the (generally only densely defined) 
linear operator $S$ on $C_0 (\g)$ (or $L^\infty(\g)$) satisfying the identity $(S \otimes \mathrm{id} )W = W^*$ informally.  
The antipode admits a polar decomposition $S = R\circ\tau_{-i/2}$ where $R$ is an antiautomorphism 
of $L^\infty (\g )$
(the {\it unitary antipode}) and $\{ \tau_t \}_{t\in\R}$ is a one-parameter
group of automorphisms (the {\it scaling group}).
In the case of a locally compact group, the scaling group is trivial
and the antipode is the antiautomorphism sending a function $f \in C_0(G)$ to the function $s\mapsto f(s^{-1} )$.  
Using the antipode $S$ one can endow the convolution algebra $L^1(\g)$ with a densely defined 
involution by considering the norm-dense subalgebra $L^1_\sharp (\g)$ of $L^1(\g)$ consisting
of all  $w\in L^1(\g)$ for which there exists an $w^\sharp \in L^1(\g)$ with
$\langle w^\sharp, x\rangle = \overline{ \langle w, S(x)^* \rangle}$ for each $x\in \mathcal{D}(S)$.
It is known from  \cite{Ku01} and Section~2 of \cite{KuVa03} that $L^1_\sharp(\g)$ 
is an involutive Banach algebra with involution  $w\mapsto w^\sharp$  
and norm $\|w\|_\sharp = \mbox{max}\{\|w\|, \|w^\sharp\|\}$.  

Associated to any locally compact quantum group $\g$ is its 
{\it dual locally compact quantum group} $\widehat{\g}$, whose associated algebras, 
coproduct, and fundamental unitary are given by 
$C_0(\widehat \g) = 
\overline{\{(w \otimes \mathrm{id})W: w \in B(L^2(\g))_*\}}^{\|\cdot\|} \subseteq B (L^2(\g))$, 
$L^\infty(\widehat{\g}) = C_0(\widehat{\g})''$ in $B(L^2(\g))$, $\hat \Delta (x) = \hat{W}^*(1\otimes x)\hat W$, 
and $\hat W = \Sigma W^* \Sigma$.  
Then in fact $W \in M(C_0(\g) \otimes C_0(\widehat{\g}))$,
and the {\it Pontryagin duality theorem} asserts that  
the bidual quantum group $\widehat{\widehat{\g}}$ is canonically identified with the original 
quantum group $\g$.  One says that  a locally compact quantum group $\g$ 
is {\it compact} if $C_0 (\g )$ is unital, and {\it discrete} if $\widehat{\g}$ is compact, 
which is equivalent to $C_0 (\g )$ being a direct sum of matrix algebras. 

For a locally compact quantum group $\g$, we can always assume that the left and 
right Haar weights are related by $\psi = \varphi \circ R$, where $R$ is the unitary antipode.  
If the left and right Haar weights $\varphi$ and $\psi$ of $\g$ coincide then we say 
that $\g$ is {\it unimodular}. In general, the failure of $\psi$ to be left-invariant 
is measured by the {\it modular element}, which is a strictly positive element $\delta$ 
affiliated with $L^\infty(\g)$ satisfying the identities  $\Delta(\delta) = \delta \otimes \delta$ 
and $\displaystyle \psi(\cdot) = \varphi(\delta^{\frac{1}{2}}\cdot \delta^{\frac{1}{2}})$.    
Compact quantum groups are always unimodular, and the corresponding Haar weight can always 
be chosen to be a state. Although discrete groups are always unimodular, 
discrete quantum groups need not be. We recall that a locally compact quantum group $\g$ 
is said to be of {\it Kac type} (or a {\it Kac algebra}) if $\g$ has trivial scaling group, and $R\sigma_t =\sigma_{-t}R$, where $(\sigma_t)_t$ is the modular automorphism group associated to  $\varphi$.  For discrete quantum groups $\g$, being of Kac type is equivalent to the traciality of the Haar state on $\widehat{\g}$.

The Fourier transform on $L^1(\g)$ is given by $\lambda=\mathcal{F}:L^1(\g)\rightarrow C_0(\widehat{\g}), w\mapsto (w\otimes \mathrm{id})(W)$. Also, it extends to an onto isometry $\mathcal{F}_2:L^2(\g)\rightarrow L^2(\widehat{\g})$. More precisely, $\mathcal{I}=\left \{x\in \mathfrak{n}_{\varphi}\Bigr| \exists _x\varphi \in L^1(\g)~\langle y^*,_x \varphi\rangle_{L^{\infty}(\g),L^1(\g)}=\langle x,y \rangle_{L^2(\g)} ~\forall y\in \mathfrak{n}_{\varphi}\right\}$ and $\left \{ \lambda(_x \varphi)\Bigr| x\in \mathcal{I}\right\}$ form norm-dense cores for $L^2(\g)$ and $L^2(\widehat{\g})$ respectively and for any $x\in \mathcal{I}$,
$ \left \| x\right\|_{L^2(\g)}=\left \| \lambda(_x\varphi) \right \|_{L^2(\widehat{\g})}$ by definition of the dual Haar weight $\widehat{\varphi}$.

For an element $\xi$ of an Hilbert space $H$, we will often use Bra-ket notation $\langle \xi|\in B(H,\Comp)$ and $|\xi\rangle\in B(\Comp,H)$ defined by
\[\langle \xi|:\eta\mapsto \langle \eta,\xi\rangle_H~\mathrm{for~all~}\eta\in H\mathrm{~and~}|\xi\rangle: z\mapsto z\xi~\mathrm{for~all~}z\in \Comp.\]
In particular, $\langle \xi|\eta\rangle=\langle \eta,\xi\rangle_H$ for all $\xi,\eta\in H$. Also, we denote by $\Sigma:H\otimes H\rightarrow H\otimes H$ the swap operator $\xi_1\otimes \xi_2\mapsto \xi_2\otimes \xi_1$.

\subsection{Completely bounded representations and corepresentation operators}

Let $H$ be a fixed Hilbert space.  Recall that there is a bijective correspondence between completely bounded representations $\pi:L^1(\g) \to B(H)$ and operators $ V \in L^\infty(\g) \overline{\otimes} B(H) \subseteq B(L^2(\g) \otimes H)$ satisfying the identity $(\Delta \otimes \mathrm{id})V = V_{13}V_{23}$.  Such operators $V$ are called {\it corepresentations}.   The association $\pi \longleftrightarrow V$ is given by 
\[
\pi(w) = (w \otimes \mathrm{id})V \qquad (w \in L^1(\g)),
\] 
and we have $\|\pi\|_{cb} = \|V\|$.  
We call a representation $\pi: L^1(\g) \to B(H)$ a $\ast$-representation if its restriction to the $\ast$-subalgebra $L^1_\sharp(\g)$ is involutive in the usual sense.  In this case, $\pi$ is automatically completely contractive.  There is a bijective correpondence between non-degenerate $\ast$-representations of $L^1(\g)$ and unitary corepresentations.  Moreover, any unitary corepresentation $V$ actually belongs to the multiplier algebra $M(C_0(\g) \otimes  K(H)) \subseteq L^\infty(\g) \overline{\otimes} B(H)$.   Two representations $\pi:L^1(\g) \to B(H_\pi)$ and $\sigma:L^1(\g) \to B(H_\sigma)$ are called {\it similar} (or equivalent) if there exists an invertible $T \in B(H_\pi, H_\sigma)$ such that $\sigma = T \circ \pi (\cdot) \circ T^{-1}$.  At the level of corepresentations, this is equivalent to saying that $V_\sigma = (\mathrm{id} \otimes T) V_{\pi}(\mathrm{id} \otimes T^{-1})$.  We say that a unitary corepresentation $V \in M(C_0(\g) \otimes  K(H))$ is {\it irreducible} if $\{T \in B(H): T\pi(\cdot) = \pi(\cdot) T\} = \{T \in B(H): (1 \otimes T)V = V(1 \otimes T)\} = \mathbb  C 1$.

Keeping in line with what is now standard terminology, we will often refer to unitary corepresentations $V \in M(C_0(\g) \otimes K(H))$ (and also the corresponding $\ast$-representations $\pi:L^1(\g) \to B(H)$) as {\it unitary representations} of $\g$.

\begin{remark} \label{non-cb}
As mentioned in the introduction, not every bounded representation $\pi:L^1(\g) \to B(H)$ is automatically completely bounded \cite{ChSa13, BrDaSa13}.  For such representations, there does not exist a corresponding (bounded) corepresentation operator $V \in L^\infty(\g) \overline{\otimes}B(H)$
\end{remark}

\subsection{Amenability and co-amenability}  \label{amen-coamen}

We recall here the basic terminology and facts on (co-)amenability for quantum groups. 

\begin{definition}
\begin{enumerate}
\item A locally compact quantum group $\g$ is called amenable if there exists a state $m\in L^{\infty}(\g)^*$ such that 
\[m(w\otimes \mathrm{id})\Delta=w(1)m~\mathrm{for~all~}w\in L^1(\g).\]
We call such a state $m$ a left-invariant mean on $L^{\infty}(\g)$.
\item A locally compact quantum group $\g$ is called co-amenable if there exists a state $\epsilon:C_0(\g)\rightarrow \Comp$ such that
\[(\mathrm{id}\otimes \epsilon)\Delta=\mathrm{id}_{C_0(\g)}.\]
Such a state $\epsilon$ is called a co-unit for $C_0(\g)$. Equivalently, co-amenability is defined as the existence of a net $(\xi_j)_j\subseteq L^2(\g)$ of unit vectors such that
\[\lim_j \left \|W(\xi_j\otimes \xi)-\xi_j\otimes \xi \right \|_{L^2(\g)\otimes L^2(\g)}\rightarrow 0~\mathrm{for~each~}\xi\in L^2(\g),\]
where $W\in B(L^2(\g)\otimes L^2(\g))$ is the multiplicative unitary.
\end{enumerate}
\end{definition}

\begin{remark}
It is well-known that co-amenability of $\widehat{\g}$ implies amenability of $\g$ \cite{BeTu03}. For a net $(\xi_j)_j \subseteq L^2(\g)$ such that $\displaystyle \lim_j \left \|\widehat{W}(\xi_j\otimes \xi)-\xi_j\otimes \xi\right\|_{L^2(\g)\otimes L^2(\g)}=0$ for each $\xi\in L^2(\g)$, we may assume that the net $(\widehat{w}_{\xi_j,\xi_j})_j$ converges to $\displaystyle \lim_j \widehat{w}_{\xi_j,\xi_j}\in B(L^2(\g))^*$  with respect to the weak $*$-topology (thanks to the Alaoglu's theorem). Then one has
\[\lim_j \widehat{w}_{\xi_j,\xi_j}\Bigr |_{C_0(\widehat{\g})}=\widehat{\epsilon}~\mathrm{and}~\lim_j \widehat{w}_{\xi_j,\xi_j}\Bigr|_{L^{\infty}(\g)}=m,\]
where $\widehat{\epsilon}$ is the co-unit for $C_0(\widehat{\g})$ and $m$ is a left invariant mean on $L^{\infty}(\g)$.

Indeed, for any $w\in L^1(\g)$ and $x\in L^{\infty}(\g)$ we have
\begin{align*}
\lim_j \widehat{w}_{\xi_j,\xi_j}((w\otimes \mathrm{id})(\Delta(x)))&=\lim_j (w\otimes \widehat{w}_{\xi_j,\xi_j})(W^*(1\otimes x)W)\\
&=\lim_j (\widehat{w}_{\xi_j,\xi_j}\otimes w)(\widehat{W}(x\otimes 1)\widehat{W}^*)\\
&=\lim_j \widehat{w}_{\xi_j,\xi_j}(x)w(1).
\end{align*}
\end{remark}

\subsection{Crossed products %$L^{\infty}(N)\rtimes_{\alpha}H$
 as locally compact quantum groups}

In this subsection we briefly recall how the von Neumann algebraic crossed product $L^{\infty}(N)\rtimes_{\alpha} H$ is understood as a locally compact quantum group, where $\alpha:H\rightarrow \mathrm{Aut}(N)$ is a continuous group homomorphism. Given any such $\alpha$, there always exists a group homomorphism $t:H\rightarrow (0,\infty)$ such that
\[\int_N f(x)dx=\int_N f(\alpha_h(x))t(h)dx\]
for all $h\in H$ and $f\in L^1(N)$.

Given an action $\alpha$, we also denote by $\alpha:L^{\infty}(N)\rightarrow L^{\infty}(H\times N)$ the corresponding $\ast$-homomorphism defined by $(\alpha(g))(h,n)=g(\alpha_h(n))$ for all $h\in H$ and $n\in N$. On the von Neumann algebraic crossed product
\[L^{\infty}(\g)=L^{\infty}(N)\rtimes_{\alpha}H =(\alpha( L^{\infty}(N))\cup (VN(H)\otimes 1))'',\]
there exists a natural multiplicative unitary $W \in B(L^2(H \times N))$, making  $L^{\infty}(N)\rtimes_{\alpha} H$ into a locally compact quantum group, which is given by
\[(W^*(f))(h_1,n_1,h_2,n_2)=f(h_2^{-1}h_1,n_1,h_2,\alpha_{h_2^{-1}h_1}(n_1)n_2) \qquad h_1,h_2\in H, \ n_1,n_2\in N.\]
Then $\Delta(x)=W^*(1\otimes x)W$ gives the comultiplication and $\g=(L^{\infty}(\g),\Delta)$ turns out to be of Kac type (see [Corollary 3.6.17, \cite{Va01}]). In this case, $\tau:L^{\infty}(H\times N)\rightarrow L^{\infty}(H\times N)$, given by
\[(\tau(f))(h,n)=f(h,\alpha_h(n)),\]
is a $*$-automorphism, but generally it is not an isometry on $L^2(H\times N)$. However, using the function $h\mapsto t(h)$, we are able to get an isometry $\tau_2:L^2(H\times N)\rightarrow L^2(H\times N)$ defined by
\[(\tau_2(f))(h,n)=f(h,\alpha_h(n))t(h)^{\frac{1}{2}}.\]

On the dual side, the underlying von Neumann algebra is
\[L^{\infty}(\widehat{\g})=L^{\infty}(H)\overline{\otimes} VN(N)\]
and the corresponding multiplicative unitary is given by $\widehat{W}=\Sigma W^*\Sigma $.

Thanks to \cite{Ng02} and \cite{DQV02}, we have the following characterization of amenability for the dual of the crossed product $\g=(L^{\infty}(N)\rtimes_{\alpha}H,\Delta)$:
\[\begin{array}{lll} \widehat{\g}~\mathrm{amenable}&\iff H\mathrm{~amenable}&\iff \g~\mathrm{co-amenable}. \end{array}\]
Moreover, we can give explicit descriptions of a net satisfying

\[ \lim_k \left \| W^*(\xi_k\otimes \xi)-\xi_k\otimes \xi \right \|_{L^2(H\times N\times H\times N)}=0 ~\forall \xi\in L^2(H\times N) .\]

\begin{proposition}\label{prop:bai}

Let $N$ be discrete, $H$ be amenable and choose a net $(f_i)_i\subseteq L^2(H)$ such that 
\[\lim_i \int_H \left |f_i(h_0h)-f_i(h) \right |^2 dh =0\]
uniformly for $h_0$ on compact subsets of $H$. Then 
\[\lim_{i}\left \|W^*(f_i\otimes \chi_{e_N}\otimes \xi)-f_i\otimes \chi_{e_N}\otimes \xi \right \|_{L^2(H\times N\times H\times N)}=0\] 
for each $\xi\in L^2(H\times N)$.

\end{proposition}
\begin{proof}
We may assume that $\xi\in C_c(H\times N)$. Since $\displaystyle \int_{H}\left |f_i(h_2^{-1}h_1)-f_i(h_1)\right |^2 dh_1$ converges to $0$ uniformly for $h_2$ on $supp(\xi)$,
\begin{align*}
&\lim_{i} \left \|W^*(f_i\otimes \chi_{e_N}\otimes \xi)-f_i\otimes \chi_{e_N}\otimes \xi \right\|_{L^2(H\times N\times H\times N)}^2\\
&=\lim_i \int_{ H\times H\times N}\left |f_i(h_2^{-1}h_1)\xi(h_2,n_2)-f_i(h_1)\xi(h_2,n_2)\right |^2  dh_1dh_2dn_2\\
&=\lim_i \int_{H\times N}\left|\xi(h_2,n_2)\right |^2 \int_H \left |f_i(h_2^{-1}h_1)-f_i(h_1)\right |^2  dh_1dh_2dn_2=0.
\end{align*}

\end{proof}

\subsection{Completely bounded similarity degree}\label{pre:degree}

In \cite{Pi98}, G. Pisier analyzed the notion of ``similarity degree'' for completely bounded representations of completely contractive Banach algebras in relation to the Kadison similarity problem and Dixmier's problem for discrete groups.  In Pisier's original work, there were certain assumptions made on the existence of units in the algebras under consideration, and later N. Spronk verified that Pisier's techniques work in general \cite{Sp02}. Let us collect some results of \cite{Pi98} and \cite{Sp02} that are necessary for our work.

\begin{definition}
Let $\A$ be a completely contractive Banach algebra and suppose that $\A$ admits at least one injective completely contractive representation $\lambda:\A\rightarrow B(H_{\lambda})$.
\begin{enumerate}
\item We say that $\A$ has the completely bounded similarity property if every completely bounded homomorphism $\pi:\A\rightarrow B(H_{\pi})$ admits an invertible $T\in B(H_{\pi})$ such that $T\circ \pi(\cdot)\circ T^{-1}$ is completely contractive.

\item Suppose that a completely contractive Banach algebra $\A$ has the completely bounded similarity property. The completely bounded similarity degree $d_{cb}(\A)$ is defined as the infimum of $\alpha\in (0,\infty)$ satisfying that every completely bounded homomorphism $\pi:\A\rightarrow B(H_{\pi})$ admits an invertible $T\in B(H_{\pi})$ such that
\begin{align*}
(a)&T\circ \pi(\cdot)\circ T^{-1}~\mathrm{is~completely~contractive}\\
(b)&\left \|T\right\|\left \|T^{-1}\right\|\leq K\left \|\pi\right\|_{cb}^{\alpha}~\mathrm{for~some~universal~constant~}K>0.
\end{align*}
\end{enumerate}
\end{definition}

\begin{remark}
For every completely contractive Banach algebra with the completely bounded similarity property, the existence of such $\alpha\in [1,\infty)$ is known. Moreover, the completely bounded similarity degree $d_{cb}(\A)$ is always a natural number.
\end{remark}

Let us now take $\A=L^1(\g)$ with the convolution product $\star$ and let $c$ be the smallest cardinality of a dense subset of $L^1(\g)$. For each $a\geq 1$, define 
$\mathrm{Hom}_{a}$ as the set of all non-degenerate homomorphisms $\pi:L^1(\g)\rightarrow B(H_{\pi})$ with $\left \|\pi \right\|_{cb}\leq a$ and $dim(H_{\pi})\leq c$.

We equip $L^1(\g)$ with the norm structure
\[\left \|x\right\|_a=\sup_{\pi\in \mathrm{Hom}_a}\left \|\pi(x)\right\|_{B(H_{\pi})}\]
for all $x\in L^1(\g)$ and define $\displaystyle \widetilde{L^1(\g)}_a$ as the completion of $L^1(\g)$ with respect to the norm $\left \|\cdot \right \|_a$. From now on we consider $\widetilde{L^1(\g)}_a$ as a subalgebra of $\displaystyle \bigoplus_{\pi\in \mathrm{Hom}_a}B(H_{\pi})$ in the obvious way, and equip it with the natural operator subspace structure coming from this inclusion.

We denote by $\iota_{a}:L^1(\g)\hookrightarrow \widetilde{L^1(\g)}_a$ the natural embeddings and define multiplication maps
\[m_{N,a}:L^1(\g)^{N\otimes_h}\rightarrow \widetilde{L^1(\g)}_a,~x_1\otimes \cdots \otimes x_N\mapsto \iota_a(x_1\star \cdots \star x_N)\]
and
\[m_{N}:L^1(\g)^{N\otimes_h}\rightarrow C_0(\widehat{\g}),~x_1\otimes \cdots \otimes x_N\mapsto \lambda(x_1\star \cdots \star x_N),\]
where $\otimes_h$ denotes the Haagerup tensor product.  These  maps are completely bounded with $\left \|m_{N,a}\right \|_{cb}\leq a^N$ and $\left \|m_N\right \|_{cb}\leq 1$.

Let us suppose that $\widehat{\g}$ is co-amenable and $\A=L^1(\g)$ has the Day-Dixmier property with $d_{cb}(L^1(\g))\leq \gamma$. Then $\widetilde{\A}_1=C_0^u(\widehat{\g})=C_0(\widehat{\g})$ since every completely contractive representation is a $\ast$-representation and $\iota_a:L^1(\g)\hookrightarrow \widetilde{L^1(\g)}_a$ extends to a completely bounded map $j_a:C_0(\widehat{\g})\rightarrow \widetilde{L^1(\g)}_a$, $\lambda(f)\mapsto \iota_a(f)$, with $\left \|j_a\right\|_{cb}\leq Ka^{\gamma}$.

\begin{proposition}\label{prop1}

Suppose that $\widehat{\g}$ is co-amenable, $\g$ has the Day-Dixmier property and $m_N:L^1(\g)^{N\otimes_h}\rightarrow C_0(\widehat{\g})$ is a complete surjection, i.e. there exists $K>0$ such that 
\[\mathrm{Ball}(M_n(C_0(\widehat{\g})))\subseteq K (id_n\otimes m_{N})(\mathrm{Ball}(M_n(L^1(\g)^{N\otimes_h})))\mathrm{~for~all~}n\in \n.\] 
Then, for any completely bounded representation $\pi:L^1(\g)\rightarrow B(H_{\pi})$, there exists an invertible $T\in B(H_{\pi})$ such that 
\begin{align*}
\mathrm{(a)}&T\circ \pi(\cdot)\circ T^{-1}~\mathrm{is~a~}*\mathrm{-representation}\\
\mathrm{(b)}&\left \| T\right\|\left \|T^{-1}\right\|\leq K\left \|\pi\right\|_{cb}^N.
\end{align*}

\end{proposition}

\begin{proof}
Since $j_a\circ m_N=m_{N,a}:L^1(\g)^{N\otimes_h}\rightarrow \widetilde{L^1(\g)}_a$ ,
\[\left \| j_a \right\|_{cb}\leq K \left \| j_a\circ m_{N} \right\|_{cb}=K \left \|m_{N,a} \right\|_{cb}\leq Ka^N.\]

Now, for any completely bounded representation $\pi:L^1(\g)\rightarrow B(H_{\pi})$, the extension $\widetilde{\pi}:\widetilde{L^1(\g)}_a\rightarrow B(H_{\pi})$ exists as a completely contractive homomorphism for $a=\left \|\pi \right\|_{cb}$. Then $\widetilde{\pi}\circ j_a:C_0(\widehat{\g})\rightarrow B(H_{\pi})$ has completely bounded norm less than $Ka^N$. Since $C_0(\widehat{\g})$ is of course an operator algebra, there exists an invertible $T\in B(H_{\pi})$ such that 
\begin{align*}
\mathrm{(a)}'&T\circ [\widetilde{\pi}\circ j_a(\cdot)]\circ T^{-1}~\mathrm{is~complete~contractive~and}\\
\mathrm{(b)}'&\left \| T\right\|\left \|T^{-1}\right\|\leq Ka^N.
\end{align*}

Moreover, $T\circ [\widetilde{\pi}\circ j_a(\cdot)]\circ T^{-1}$ is a $*$-homomorphism since every contractive homomorphism on a $C^*$-algebra is automatically a $*$-homomorphism \cite{Pa84}. Finally we have that
\begin{align*}
\mathrm{(a)}&T\circ \pi(\cdot)\circ T^{-1}=(T\circ [\widetilde{\pi}\circ j_a(\cdot)]\circ T^{-1})\circ \lambda~\mathrm{is~a~}*\mathrm{-representation~and}\\
\mathrm{(b)}&\left \|T\right\| \left \|T^{-1}\right\| \leq K\left \|\pi \right\|_{cb}^N.
\end{align*}

\end{proof}

\section{Compact quantum groups without the Day-Dixmier property} \label{section:nonKac}

In this section, we will establish that the Day-Dixmier property does not generally hold within the category of compact quantum groups. In other words, the role of associated anti-representation $\check{\pi}$ highlighted in \cite{BrDaSa13} is indispensable. We begin by recalling some facts about unitary representations of compact quantum groups.

Let $\g$ be a compact quantum group, and denote by $\text{Irr}(\g)$ the collection of equivalence classes of irreducible unitary representations of $\g$ under the relation of unitary equivalence. For each $\alpha \in \text{Irr}(\g)$ we fix a representative $u^\alpha \in M(C_0(\g) \otimes K(H_\alpha)) = C_0(\g) \otimes B(H_\alpha)$.  We write $n_\alpha = \dim H_\alpha < \infty$ for the {\it dimension of $\alpha$}.  By fixing an orthonormal basis $(e_j)_{1 \le j \le n_\alpha} \subset H_\alpha$, we can then write $u^\alpha = [u^\alpha_{ij}] \in M_{n_\alpha}(C_0(\g))$.  For each $\alpha \in \text{Irr}(\g)$, there exists a positive invertible $Q_\alpha \in B(H_\alpha)$ with the properties that  $\text{Tr}(Q_\alpha) = \text{Tr}(Q_\alpha^{-1})$ and $Q_\alpha^{\frac{1}{2}} \overline{u^\alpha} Q_\alpha^{-\frac{1}{2}}$ is a unitary irreducible representation of $\g$ \cite{Ti08}.  The quantity $d_\alpha = \text{Tr}(Q_\alpha) = \text{Tr}(Q_\alpha^{-1})$ is called the {\it quantum dimension of $\alpha$}.
 
Now let us suppose that $\g$ is a compact quantum group with the Day-Dixmier property, i.e. (a) every completely contractive representation of $L^1(\g)$ is a $*$-representation and (b) every completely bounded representation of $L^1(\g)$ is similar to a $*$-representation (which is automatically completely contractive \cite{BrDaSa13}). Then due to \cite [Theorem 4.2.8]{Sp02} and \cite[Corollary 2.4]{Pi98}, there exists $K,\gamma >0$ with the property that every completely bounded representation $\pi:L^1(\g)\rightarrow B(H_{\pi})$ admits an invertible $T\in B(H_{\pi})$ such that $T\circ \pi(\cdot )\circ T^{-1}$ is a $*$-representation and $\left \|T\right \| \left \|T^{-1}\right \|\leq K \left \|\pi\right\|_{cb}^{\gamma}$.

Let $\overline{u^{\alpha}}=(u_{i,j}^*)_{1\leq i,j\leq n_{\alpha}}\in L^{\infty}(\g)\overline{\otimes}M_{n_{\alpha}}$ be associated with a completely bounded representation $\pi_{\alpha}:L^1(\g)\rightarrow M_{n_{\alpha}}$ satisfying 
\[\left \|\pi_{\alpha} \right\|_{cb}=\left \|\overline{u_{\alpha}}\right \|_{L^{\infty}(\g)\overline{\otimes }M_{n_{\alpha}}}\leq n_{\alpha}^2.\]

Then there exists $T_{\alpha}\in M_{n_{\alpha}}$ such that $\left \|T_{\alpha}\right \| \left \|T_{\alpha}^{-1}\right\|\leq K\cdot n_{\alpha}^{2\gamma}$ and $T_{\alpha} \overline{u_{\alpha}} T_{\alpha}^{-1}$ is a unitary irreducible corepresentation and we know that $Q_{\alpha}^{\frac{1}{2}}\overline{u_{\alpha}}Q_{\alpha}^{-\frac{1}{2}}$ is also a unitary irreducible corepresentation. By Schur's lemma and the assumption of irreducibility, for each $\alpha\in \mathrm{Irr}(\g)$, there exists a unitary $U_{\alpha}\in M_{n_{\alpha}}$ and a constant $c_{\alpha}\in \Comp$ such that
\[U_{\alpha}T_{\alpha}=c_{\alpha}\cdot Q_{\alpha}^{\frac{1}{2}}.\]

We denote by $\lambda_{min}^{\alpha}$ and $\lambda_{max}^{\alpha}$ the smallest and largest eigenvalues of $Q_{\alpha}$ respectively for each $\alpha\in \mathrm{Irr}(\g)$. Then we have
\[\sqrt{\frac{\lambda^{\alpha}_{max}}{\lambda^{\alpha}_{min}}}=\left \|T_{\alpha} \right \| \left \|T_{\alpha}^{-1}\right \|\leq Kn_{\alpha}^{2\gamma}~\mathrm{for~all~}\alpha\in \mathrm{Irr}(\g).\]

This implies that 
\[d_{\alpha}\leq n_{\alpha}\lambda^{\alpha}_{max}\leq n_{\alpha}\frac{\lambda^{\alpha}_{max}}{\lambda^{\alpha}_{min}}\leq K^2n_{\alpha}^{4\gamma+1}~\mathrm{for~all~}\alpha\in \mathrm{Irr}(\g).\]

\begin{theorem} \label{thm:subexponential}
Let $\g$ be a compact quantum group satisfying the Day-Dixmier property and suppose that the function $\alpha\mapsto n_{\alpha}$ has subexponential growth, i.e.
\[\limsup_{n\rightarrow \infty}~ \Big(\sum_{\alpha\in \mathrm{Irr}(\g):u^{\alpha}\leq v^{n\otimes }}n_{\alpha}^2\Big)^{\frac{1}{n}}=1\]
for any finite dimensional representation $v$. 
Then $\g$ is of Kac type.
\end{theorem}

\begin{proof}

By [Corollary 4.5, \cite{DPR16}], it is sufficient to show that the function $\alpha\mapsto d_{\alpha}$ has the subexponential growth.

For any finite dimensional representation $v$, we have
\begin{align*}
\sum_{\alpha\in \mathrm{Irr}(\g):u^{\alpha}\leq v^{n\otimes }}d_{\alpha}^2&\leq \sum_{\alpha\in \mathrm{Irr}(\g):u^{\alpha}\leq v^{n\otimes }} K^4n_{\alpha}^{8\gamma+2}\\
&\leq K^4(\sum_{\alpha\in \mathrm{Irr}(\g):u^{\alpha}\leq v^{n\otimes }} n_{\alpha}^2 )^{4\gamma+1}.
\end{align*}

Therefore,
\[\limsup_{n\rightarrow \infty}~(\sum_{\alpha\in \mathrm{Irr}(\g):u^{\alpha}\leq v^{n\otimes }}d_{\alpha}^2)^{\frac{1}{n}}\leq  \limsup_{n\rightarrow \infty}~K^{\frac{4}{n}}(\sum_{\alpha\in \mathrm{Irr}(\g):u^{\alpha}\leq v^{n\otimes }}n_{\alpha}^2)^{\frac{4\gamma+1}{n}}=1\]
for any finite dimensional representation $v$. Hence we reach the conclusion.

\end{proof}

\begin{corollary} \label{cor:Gq}
Let $G$ be a simply connected semisimple compact Lie group. Then the Drinfeld-Jimbo $q$-deformations $G_q$ with $0<q<1$ does not have the Day-Dixmier property.
\end{corollary}
\begin{proof}

For $\g=G_q$ with $0<q<1$, the function $\alpha\mapsto n_{\alpha}$ has polynomial growth \cite[Theorem 2.4.7]{NeTu13}.
\end{proof}

\subsection{An explicit example} Despite the applicability of the above theorem to many concrete examples of compact quantum groups, we find ourselves unable at the present time to construct so many explicit examples of completely bounded representations $\pi:L^1(\g) \to B(H_{\pi})$ that fail to be similar to $\ast$-representations, even for the simplest $q$-deformations, like Woronowicz's $SU_q(2)$ quantum group.  Let us content ourselves for the time being with at least one explicit example, obtained from an infinte tensor product of $SU_q(2)$'s.   

Let $\g=\displaystyle \prod_{n\in \n}SU_{q_n}$ with $q_n\rightarrow 0$ as $n\rightarrow \infty$ and denote by $a_n$ and $c_n$ the standard generators of $SU_{q_n}$ in $\g=\displaystyle \prod_{n\in \n}SU_{q_n}$. Then 
\[V:=\bigoplus_{n\in \n}\left ( \begin{array}{cc}  a_n&-q_nc_n^*\\ c_n&a_n^* \end{array} \right )\in  L^{\infty}(\g)\overline{\otimes}(\ell^{\infty}-\bigoplus_{n\in \n}M_2)\]
is a representation of $\g$, so that its contragredient 
\[\overline{V}=\bigoplus_{n\in \n}\left ( \begin{array}{cc}  a_n^*&-q_nc_n\\ c_n^*&a_n \end{array} \right )\]
also satisfies $(\Delta\otimes \mathrm{id})\overline{V}=\overline{V}_{13}\overline{V}_{23}$. Note that

\begin{align*}
\left \|\overline{V}\right\|_{ L^{\infty}(\g)\overline{\otimes}(\ell^{\infty}-\bigoplus_{n\in \n}M_2)}&=\sup_{n\in \n} \left \| \left ( \begin{array}{cc}  a_n^*&-q_nc_n\\ c_n^*&a_n \end{array} \right ) \right \|_{M_2(L^{\infty}(\g))}\\
&\leq \sup_{n\in \n} (\left \| a_n^*\right \|+\left \|q_nc_n \right \|+\left \|c_n^* \right \|+\left \| a_n\right \|)\\
&\leq 4.
\end{align*}
It can be readily checked that the completely bounded representation associated to $\overline{V}$ is non-degenerate.  Hence by \cite[Theorem 6.1]{BrDaSa13}, the completely bounded representation $\displaystyle \pi:L^1(\g)\rightarrow \ell^{\infty}-\bigoplus_{n\in \n}M_2$ correspondng to $\overline{V}$ will be similar to a $\ast$-representation iff $\check{\pi}$ is completely bounded iff $\overline{V}$ is invertible.  

Now, if we assume that $\overline{V}$ is invertible, then the algebraic inverse
\begin{align*}
&\bigoplus_{n\in \n}\left ( \begin{array}{cc}  q_n&0\\ 0&q_n^{-1} \end{array} \right ) \left ( \begin{array}{cc}  a_n&c_n\\ -q_nc_n^*&a_n^* \end{array} \right ) \left ( \begin{array}{cc}  q_n^{-1}&0\\ 0&q_n \end{array} \right )\\
&=\bigoplus_{n\in \n}\left ( \begin{array}{cc}  a_n&q_n^2 c_n\\ -q_n^{-1}c_n^*&a_n^* \end{array} \right )
\end{align*}
should be an element of $L^{\infty}(\g)\overline{\otimes}(\ell^{\infty}-\bigoplus_{n\in \n} M_2)$.
However, this is impossible because $\left \|q_n^{-1}c_n^* \right\|_{L^{\infty}(\g)}=q_n^{-1}\rightarrow \infty$ as $n\rightarrow \infty$. This implies that the completely bounded representation $\displaystyle \pi:L^1(\g)\rightarrow \ell^{\infty}-\bigoplus_{n\in \n}M_2$ correspondng to $\overline{V}$ is not similar to a $*$-representation.

\section{New examples of amenable Kac-type quantum groups with the Day-Dixmier property} \label{section:positive}

The results of Section \ref{section:nonKac} imply that, if we want to classify the amenable locally compact quantum groups with the Day-Dixmier property, it is reasonable to first restrict our attention to the framework of Kac algebras. The main purpose of this section is to exhibit several new classes of examples that do have the Day-Dixmier property. In particular, we will establish the affirmative answer on all of amenable discrete quantum groups of Kac type and the duals of certain crossed products.

As a first step in this direction, we will show that the idea of [Theorem 6.2, \cite{BrDaSa13}] (which shows that if $\g$ is compact and of Kac type, then any completely bounded representation $\pi:L^1(\g) \to B(H_\pi)$ is similar to a $\ast$-representation without any a priori complete boundedness assumptions on $\check{\pi}$) is still valid for the much wider class of examples where $\g$ is of Kac type, the left Haar weight $\varphi$ is tracial and $\widehat{\g}$ is co-amenable. Of course, in this case, the antipode $S$ coincides with the unitary antipode $R$ and $R\circ *=*\circ R$. Also, there exists a modular element $\delta$ that is positive element affiliated to the center of $L^{\infty}(\g)$ such that $\varphi(x)=\varphi(R(x)\delta).$

\begin{theorem} \label{thm:tracial-mean}

Let $\g$ be a locally compact quantum group of Kac type such that $\widehat{\g}$ is co-amenable and the left Haar weight $\varphi$ is tracial. Suppose that the net $(\xi_j)_j$ (coming from the definition of co-amenability in Section \ref{amen-coamen}) is chosen to be $(\xi_j)_j\subseteq \mathfrak{n}_{\varphi}\cap  Z(L^{\infty}(\g))$ where $Z(\cdot)$ denotes the center. Then for any completely bounded representation $\pi:L^1(\g)\rightarrow B(H_{\pi})$  there exists an invertible $T\in B(H_{\pi})$ such that 
\begin{align*}
(a)~& T\circ \pi(\cdot)\circ T^{-1}\mathrm{~is~a~}*-\mathrm{representation}\\
(b)~& \left \|T\right \|\left \|T^{-1}\right \|\leq \left \|\pi\right \|_{cb}^4.
\end{align*}
\end{theorem}

\begin{proof}
We will use the notation and adapt the methodology presented in \cite[Theorem 6.1]{BrDaSa13} to the cases under consideration.  Given a completely bounded representation $\pi:L^1(\g)\rightarrow B(H_{\pi})$, let us define a homomorphism $\Phi:\lambda(L^1(\g))\subseteq C_0(\widehat{\g})\rightarrow B(H_{\pi}), \lambda(w)\mapsto \pi(w)$.  Our goal is to show that $\Phi$ is bounded with $\|\Phi\|_{C_0(\widehat{\g}) \to B(H_\pi)} \le \|\pi\|_{cb}^2$.  

For any $\alpha,\beta\in H_{\pi}$ and $w\in L^1(\g)$, as in the proof of \cite[Theorem 4.5]{BrDaSa13}, we have
\begin{align*}
\left |\langle \Phi(\lambda(w))\alpha ,\beta\rangle_{H_{\pi}}\right | &= \left |\langle \pi(w)\alpha ,\beta\rangle_{H_{\pi}}\right | \\
&=\left | \langle T^{\pi}_{\alpha,\beta},w\rangle_{L^{\infty}(\g),L^1(\g)}\right | \\
&=\left | \langle  (T^{\pi}_{\alpha,\beta})^*,w^\sharp \rangle_{L^{\infty}(\g),L^1(\g)}\right |\\
&=\lim_j \left |\langle (T^{\pi}_{\alpha,\beta})^*\widehat{\lambda}(\widehat{w}_{\xi_j, \xi_j}),w^\sharp \rangle_{L^{\infty}(\g),L^1(\g)}  \right |,
\end{align*}
where $(\widehat{w}_{\xi_j, \xi_j})_j \subset L^1(\widehat{\g})$ is the bounded approximate identiy given in the theorem statement, and $T^{\pi}_{\alpha,\beta}= (\text{id} \otimes w_{\alpha,\beta})V_\pi=  (1\otimes \langle \beta |)V_{\pi}(1\otimes |\alpha \rangle) \in L^\infty(\g)$ is the coefficient function of the corepresentation $V_{\pi}\in L^{\infty}(\g)\overline{\otimes} B(H_{\pi})$  associated with the given completely bounded representation $\pi:L^1(\g)\rightarrow B(H_{\pi})$.

Next, we show the existence of the functionals
\[\widehat{w}_j = \sum_i \widehat{w}_{a_i\xi_j,R(b_i^*)\xi_j}\in L^1(\widehat{\g})\] which have the property that $(T^{\pi}_{\alpha,\beta})^*\widehat{\lambda}(\widehat{w}_{\xi_j, \xi_j}) = \widehat{\lambda}(\widehat {w}_j)$.  Here, $(f_i)_i$ is an orthonormal basis of $H_{\pi}$, $a_i=(1\otimes \langle f_i |)V_{\pi}^*(1\otimes |\beta\rangle)$ and $R(b_i^*)=R((1\otimes \langle f_i |)V_{\pi}(1\otimes |\alpha \rangle))$. To check this, note that \[\sum_i \|\widehat{w}_{a_i\xi_j,R(b_i^*)\xi_j}\|_{L^1(\widehat\g)} \le \Big( \sum_i \left \|a_i\xi_j\right \|_{L^2(\g)}^2\Big)^{1/2} \Big(\sum_i \left \|R(b_i^*)\xi_j \right \|_{L^2(\g)}^2 \Big)^{1/2}. \] Moreover,  
\begin{align*}
\sum_i \left \|a_i\xi_j\right \|_{L^2(\g)}^2&= \sum_i \langle \xi_j |a_i^*a_i |\xi_j \rangle \\
&=\widehat{w}_{\xi_j,\xi_j}(\sum_i a_i^*a_i )\\
&=\widehat{w}_{\xi_j,\xi_j}(\sum_{i}  (1\otimes \langle \beta |) V_{\pi} (1\otimes |f_i\rangle\langle f_i|)  V_{\pi}^* (1\otimes |\beta\rangle)\\
&=(\widehat{w}_{\xi_j,\xi_j}\otimes w_{\beta,\beta})( V_{\pi}V_{\pi}^* )\\
&\leq \left \|V_{\pi}V_{\pi}^*\right \| \left \|\beta\right \|^2\leq \left \|\pi\right \|_{cb}^2\left \|\beta\right \|^2
\end{align*}
and
\begin{align*}
\sum_i \left \|R(b_i^*)\xi_j \right \|_{L^2(\g)}^2&= \sum_i \widehat{w}_{\xi_j,\xi_j}( R(b_i)R(b_i^*))\\
&=\sum_i \varphi(R(b_i)\xi_j\xi_j^* R(b_i^*))\\
&=\sum_i \varphi(\xi_j^*R(b_i^*)R(b_i)\xi_j)\\
&=(\widehat{w}_{\xi_j,\xi_j}\circ R\otimes w_{\alpha,\alpha})(V_{\pi}^*V_{\pi})\\
&\leq \left \|\pi\right\|_{cb}^2 \left \|\alpha \right \|^2.
\end{align*}

Hence $ \widehat{w}_j := \sum_i \widehat{w}_{a_i\xi_j,R(b_i^*)\xi_j}$ absolutely converges in $L^1(\widehat{\g})$ with norm less than $\left \|\pi\right \|_{cb}^2 \left \|\alpha\right \|\left \|\beta\right\|$. The fact that  $(T^{\pi}_{\alpha,\beta})^*\widehat{\lambda}(\widehat{w}_{\xi_j, \xi_j}) = \widehat{\lambda}(\widehat {w}_j)$ now follows exactly as in \cite[Theorem 4.7]{BrDaSa13}.

Therefore, we have
\begin{align*}
\left |\langle \pi(w)\alpha,\beta\rangle_{H_{\pi}}\right |&=\lim_{j}\left | \langle (T^{\pi}_{\alpha,\beta})^*\widehat{\lambda}(\widehat{w}_{\xi_j,\xi_j}),w^\sharp \rangle_{L^{\infty}(\g),L^1(\g)}\right |\\
&=\lim_{j} \left | \langle \widehat{\lambda}(\sum_i \widehat{w}_{a_i\xi_j,R(b_i)^*\xi_j}),w^\sharp\rangle_{L^{\infty}(\g),L^1(\g)} \right |\\
&=\lim_{j} \left | \langle \lambda(w)^*, \sum_i \widehat{w}_{a_i\xi_j,R(b_i)^*\xi_j} \rangle_{L^{\infty}(\widehat{\g}),L^1(\widehat{\g})} \right |\\
&\leq \left \|\lambda(w)\right \|_{C_0(\widehat{\g})}\left \|\pi\right\|_{cb}^2 \left \|\alpha \right \|\left \|\beta\right \|,
\end{align*}
which shows $\left \|\Phi\right \|\leq \left \|\pi\right\|_{cb}^2$. Finally, since $\Phi$ extends to a bounded homomorphism on $C_0(\widehat{\g})$, and $C_0(\widehat{\g})$ is a nuclear C$^*$-algebra, there exists an invertible $T\in B(H_{\pi})$ such that
\begin{align*}
(a)&~T\circ \Phi(\cdot) \circ T^{-1}~\mathrm{is~a~}*-\mathrm{homomorphism}\\
(b)&~\left \|T\right \|\left \|T^{-1}\right \|\leq \left \|\Phi\right \|^2\leq \left \|\pi\right\|_{cb}^4.
\end{align*}

Then the formula $[T\circ \Phi(\cdot)\circ T^{-1}]\circ \lambda = T\circ \pi(\cdot) \circ T^{-1}$ completes the proof.
\end{proof}

Using Theorem \ref{thm:tracial-mean} as our starting point, we now describe some new examples of quantum groups with the Day-Dixmier property.

\subsection{Example 1: Amenable discrete quantum groups of Kac type}

Outside the realm of classical amenable groups and certain duals of locally compact groups, the only subclass of truly ``quantum'' groups known to satisfy the Day-Dixmier property are the compact quantum groups of Kac type. Therefore, it is quite natural to first consider the dual setting: discrete quantum groups of Kac type. 

It is known that any co-amenable compact quantum group $\widehat{\g}$ of Kac type admits a contractive approximate identity $(\widehat{w}_{\xi_j,\xi_j})_j\subseteq B(L^2(\g))_*$ of $L^1(\widehat{\g})$ such that $\xi_j\in Z(L^{\infty}(\g))\cap L^2(\g)$.  See for example \cite[Theorem 7.3]{Br16} and \cite[Theorem 5.15]{KrRu99}.  Hence we can conclude from Theorem \ref{thm:tracial-mean} that if $\g$ is an  amenable discrete quantum group of Kac type, then $L^1(\g)$ has the similarity property with $d_{cb}(L^1(\g))\leq 4$.

\begin{remark}
Even if we suppose that $\g$ is a discrete quantum group and that $u\in M_n(L^{\infty}(\g))$ is a finite dimensional unitary representation of $\g$, it is not clear that its contragradient $u^c=\overline{u}=(u_{i,j}^*)_{1\leq i,j\leq n}$ is invertible. This question was raised in \cite{so05} and affirmatively answered by \cite{Da13}.

The Day-Dixmier property provides a generalized view on infinite dimensional representations. Let us suppose that $\g$ is an amenable discrete quantum group of Kac type and $V\in L^{\infty}(\g)\overline{\otimes}B(H)$ is a unitary representation of $\g$. Then the Day-Dixmier property implies that its contragredient $V^c$ is automatically invertible with $\left \|(V^c)^{-1}\right \|_{L^{\infty}(\g)\overline{\otimes} B(\overline{H})}\leq \left \|V^c \right\|_{L^{\infty}(\g)\overline{\otimes} B(\overline{H})}^4$ (in fact, $\leq \left \|V^c\right\|_{L^{\infty}(\g)\overline{\otimes} B(\overline{H})}^2$, as shown in the Appendix) whenever $V^c$ exists in $L^{\infty}(\g)\overline{\otimes}B(\overline{H})$.
\end{remark}

\subsection{Example 2: Some Fourier algebras of crossed products $L^{\infty}(N)\rtimes_{\alpha}H$} 

For now, we have the affirmative answer for amenable locally compact groups $G$, a large class of their duals $\widehat{G}$, compact Kac algebras and amenable discrete Kac algebras. In this subsection, we will present new examples which are non-compact, non-discrete, non-commutative and non-cocommutative in general.

Recall that for the crossed product quantum group $\g=(L^{\infty}(N)\rtimes_{\alpha} H,\Delta)$, the von Neumann algebra associated with the dual $\widehat{\g}$ is 
\[L^{\infty}(\widehat{\g})=L^{\infty}(H)\overline{\otimes} VN(N) \subseteq B(L^2(H\times N))\] 
and the left Haar weight on $L^{\infty}(\widehat{\g})$ is given by $\widehat{\varphi}=\varphi_H \otimes \widehat{\varphi}_N$ where $\varphi_H$ is the left Haar measure on $L^{\infty}(H)$ and $\widehat{\varphi}_N$ is the Plancherel weight on $VN(N)$. Note that the left Haar weight $\widehat{\varphi}$ on $L^{\infty}(\widehat{\g})$ is tracial if and only if $N$ is unimodular. 

Under the condition that $H$ is amenable and $N$ is discrete, Proposition \ref{prop:bai} tells us that a contractive approximate identity of $L^{1}(\g)$ is described by a net 
\[(w_{f_i\otimes 1_{VN(N)},f_i\otimes 1_{VN(N)}})_i=(\varphi_{L^{\infty}(H)}(\cdot f_i^2)\otimes \widehat{\varphi}_{VN(N)}(\cdot))_i,\] 
where $(w_{f_i,f_i})_i$ is a contractive approximate identity in $A(H)$, and $\varphi_{L^{\infty}(H)}$, $\widehat{\varphi}_{VN(N)}$ are left Haar weights on $L^{\infty}(H)$ and $VN(N)$ respectively. Moreover,
\[(f_i\otimes 1_{VN(N)})_i\subseteq Z(L^{\infty}(\widehat{\g}))=L^{\infty}(H)\overline{\otimes}Z(VN(N)).\]

Hence we can conclude that the Fourier algebra $A(\g)=L^1(\widehat{\g})$ of the crossed product $\g=(L^{\infty}(N)\rtimes_{\alpha} H,\Delta)$ has the Day-Dixmier property with $d_{cb}(A(\g))\leq 4$ whenever $N$ is discrete and $H$ is amenable.

\appendix
\section{Similarity degree} \label{appendix}

Under the assumption of the Day-Dixmier property for $\g$, calculating the completely bounded similarity degree $d_{cb}(L^1(\g))$ is worthy of itself since $d_{cb}(L^1(G))\leq 2$ characterizes the amenability of $G$ in the category of locally compact groups $G$. On the cocommutative side, one of the main results of \cite{LeSaSp16} is $d_{cb}(A(G))\leq 2$ for a large class of locally compact groups. In this appendix, we will show that $d_{cb}(L^1(\g))\leq 2$ whenever $\g$ is one of the following:
\begin{itemize}
\item A compact quantum group of Kac type.
\item An amenable discrete quantum group of Kac type
\item The dual of $\widehat{\g} = (L^{\infty}(N)\rtimes_{\alpha}H,\Delta)$ where $N$ is discrete and $H$ is amenable.
\end{itemize}

The key tool here is in the following theorem.

\begin{theorem}\label{thm1}
Let $\g$ be a locally compact quantum group and we fix two contractions $T_1,T_2\in B(L^2(\widehat{\g}))$. Also, suppose that there exists a contractive approximate identity $(w_{\xi_i,\eta_i})_i\subseteq L^1(\widehat{\g})$ such that 
\begin{align*}
&\lim_i \mu (\left |W(T_1\otimes \mathrm{id})W(T_2\otimes \mathrm{id})(|\xi_i\rangle\otimes 1)-|\xi_i\rangle \otimes 1 \right |^2)=0
\end{align*}
for any positive $\mu\in M(\widehat{\g})_+$. Then we have that the multiplication map $m_2:L^1(\g)\otimes_h L^1(\g)\rightarrow C_0(\widehat{\g})$ is a complete quotient map.
\end{theorem}

\begin{proof}
It is sufficient to show that the adjoint map $\Gamma=m_{2}^*:M(\widehat{\g})\rightarrow L^{\infty}(\g)\otimes_{eh}L^{\infty}(\g)$,
\[\mu\mapsto (\mathrm{id}\otimes \mathrm{id} \otimes \mu)(W_{13}W_{23}),\]
is a complete isometry.

Recall that $L^{\infty}(\g)\otimes_{eh} L^{\infty}(\g)$ is completely isometrically embedded into 
\[B(L^2(\widehat{\g}))\otimes_{eh} B(L^2(\widehat{\g}))\cong CB^{\sigma}(B(L^2(\widehat{\g})),B(L^2(\widehat{\g})))\] 
under the identification
\[\iota(A\otimes B):T\mapsto A T B~\mathrm{for~all~}A,B\in B(L^2(\widehat{\g})).\]

Hence, for any $n\in \n$ and $\mu=(\mu_{i,j})_{1\leq i,j\leq n}\in M_n(M(\widehat{\g}))$, $(\mathrm{id}_n\otimes \Gamma)(\mu)=(\Gamma(\mu_{i,j}))_{1\leq i,j\leq n}$ can be realized as an element in $CB(B(L^2(\widehat{\g})),M_n(B(L^2(\widehat{\g}))))$, which is given by
\[T\mapsto ((\Gamma(\mu_{i,j}))(T))_{1\leq i,j\leq n}.\]

First of all, a map $\Phi_{T_1,T_2}:B(L^{2}(\widehat{\g}))\rightarrow B(L^2(\widehat{\g}))$, $A\mapsto T_1 A 
T_2$ is a complete contraction for any contractions $T_1,T_2\in B(L^2(\widehat{\g}))$. Moreover, $(\mathrm{id}\otimes  \Phi_{T_1,T_2}):B(L^2(\widehat{\g}))\otimes_{eh}B(L^2(\widehat{\g}))\rightarrow B(L^2(\widehat{\g}))\otimes_{eh}B(L^2(\widehat{\g}))$ is also a complete contraction. Secondly, for any $m,n\in \n$, $\mu=(\mu_{i,j})_{1\leq i,j\leq n}\in M_n(M(\widehat{\g}))$ and $X=(x_{s,t})_{1\leq s,t\leq m} \in \mathrm{Ball}(M_m(C_0(\widehat{\g})))$, we have that
\begin{align*}
&\left \|(id_n\otimes \Gamma)(\mu)\right\|\\
&\geq \left \| [(\mathrm{id}_n\otimes \mathrm{id}\otimes \Phi_{T_1,T_2})((\mathrm{id}_n\otimes \Gamma)(\mu))](X)\right \|_{M_{mn}(B(L^2(\widehat{\g})))}\\
&= \left \| [(\mathrm{id}_n\otimes ((\mathrm{id}\otimes \Phi_{T_1,T_2})\circ \Gamma))(\mu)](X)\right \|_{M_{mn}(B(L^2(\widehat{\g})))}\\
&=\left \|((\mathrm{id}\otimes\mu_{i,j})(W( x_{s,t}\otimes \mathrm{id})( T_1\otimes \mathrm{id})W( T_2\otimes \mathrm{id})))_{1\leq i,j\leq n, 1\leq s,t\leq m}\right \|_{M_{mn}(B(L^2(\widehat{\g})))}\\
&=\left \|((\mathrm{id}\otimes \mu_{i,j})(\Sigma\widehat{\Delta}(x_{s,t})\Sigma W( T_1\otimes \mathrm{id})W( T_2\otimes \mathrm{id})))_{1\leq i,j\leq n, 1\leq s,t\leq m}\right \|_{M_{mn}(B(L^2(\widehat{\g})))}\\
&\geq \sup_k \left \|((w_{\xi_k,\eta_k}\otimes \mu_{i,j})(\Sigma\widehat{\Delta}(x_{s,t})\Sigma W( T_1\otimes \mathrm{id})W( T_2\otimes \mathrm{id})))_{1\leq i,j\leq n, 1\leq s,t\leq m}\right\|_{M_{mn}}\\
&\geq \lim_k \left \|(\mu_{i,j}((\langle \eta_k|\otimes 1)(\Sigma\widehat{\Delta}(x_{s,t})\Sigma W( T_1\otimes \mathrm{id})W( T_2\otimes \mathrm{id}))(| \xi_k\rangle \otimes 1)))_{1\leq i,j\leq n, 1\leq s,t\leq m}\right\|_{M_{mn}}
\end{align*}

Note that each $\mu_{i,j}$ is expressed as
\[\mu_{i,j}=w^1_{i,j}-w^2_{i,j}+i(w^3_{i,j}-w^4_{i,j})\]
for some positive linear functionals $w^s_{i,j}\in M(\widehat{\g})_+$, $1\leq s\leq 4$. Now, by the assumption, we know that for each $1\leq i,j\leq n, 1\leq s,t\leq m$
\begin{align*}
&\lim_k \left |\mu_{i,j}((\langle \eta_k|\otimes 1)(\Sigma\widehat{\Delta}(x_{s,t})\Sigma [W( T_1\otimes \mathrm{id})W( T_2\otimes \mathrm{id})(| \xi_k\rangle \otimes 1)- |\xi_k\rangle\otimes 1]))\right |\\
&\leq \lim_k \sum_{s=1}^4 \left |\mu_{i,j}^s ((\langle \eta_k|\otimes 1)(\Sigma \widehat{\Delta}(x_{s,t})\Sigma  [W( T_1\otimes \mathrm{id})W( T_2\otimes \mathrm{id})(| \xi_k\rangle \otimes 1)- |\xi_k\rangle\otimes 1]))\right |\\
&\leq \lim_k \sum_{s=1}^4 \left \|\mu_{i,j}^s \right \|^{\frac{1}{2}}_{M(\widehat{\g})_+}\mu_{i,j}^s(\left | W(T_1\otimes \mathrm{id})W(T_2\otimes \mathrm{id})(|\xi_k\rangle \otimes 1)-|\xi_k\rangle\otimes 1\right|^2)^{\frac{1}{2}}\\
&= 0.
\end{align*}

Therefore, we can conclude that
\begin{align*}
&\left \|(\mathrm{id}_n\otimes \Gamma)(\mu)\right\|\\
&\geq \lim_k \left \| (\mu_{i,j}((\langle \eta_k|\otimes 1)(\Sigma \widehat{\Delta}(x_{s,t})\Sigma (|\xi_k\rangle\otimes 1))))_{1\leq i,j\leq n, 1\leq s,t\leq m} \right\|_{M_{mn}}\\
&\geq \lim_k \left \| (\mu_{i,j}((1\otimes \langle \eta_k|)(\widehat{\Delta}(x_{s,t})(1\otimes |\xi_k\rangle))))_{1\leq i,j\leq n, 1\leq s,t\leq m} \right\|_{M_{mn}}\\
&=\lim_k \left \|((\mu_{i,j}\otimes w_{\xi_k,\eta_k})(\widehat{\Delta}(x_{s,t})))_{1\leq i,j\leq n, 1\leq s,t\leq m}\right\|_{M_{mn}}\\
&= \lim_k \left \|((\mu_{i,j} \star w_{\xi_k,\eta_k})(x_{s,t}))_{1\leq i,j\leq n, 1\leq s,t\leq m}\right\|_{M_{mn}}\\
&= \left \|(\mu_{i,j}(x_{s,t}))_{1\leq i,j\leq n, 1\leq s,t\leq m}\right\|_{M_{mn}}.
\end{align*}

Since $X=(x_{s,t})_{1\leq s,t\leq m}$ is arbitrary, we can see that $(id_n\otimes \Gamma)(\mu)$ is a isometry, so that $\Gamma$ is a complete isometry.
\end{proof}  

And now back to examples:

\subsection{Example 1 : compact or amenable discrete quantum groups of Kac type}

Throughout this subsection, we assume that $\g$ is a compact or amenable discrete Kac algebra. In both cases, the (left and right) Haar weight $\varphi=\psi$ on $L^{\infty}(\g)$ is tracial and the antipode $S=R$ extends to a unitary operator on $L^2(\g)$. Here we will make use of the Sweedler notation $\Delta(x)=\displaystyle \sum x_{(1)}\otimes x_{(2)}$ and the swap operator $\sigma:B(H_1\otimes H_2)\rightarrow B(H_2\otimes H_1)$, $T_1\otimes T_2\mapsto T_2\otimes T_1$.

For any $\Lambda_{\varphi}(\eta),\Lambda_{\varphi}(\eta')\in \mathfrak{n}_{\varphi}\subseteq L^2(\g)$,
\begin{align*}
&\widehat{W}(R\otimes \mathrm{id})\widehat{W}(R\otimes \mathrm{id})(\Lambda_{\varphi}\otimes \Lambda_{\varphi})(\eta\otimes \eta')\\
=&\Sigma W^* (\mathrm{id}\otimes R)W^*(\mathrm{id}\otimes R)(\Lambda_{\varphi}\otimes \Lambda_{\varphi})(\eta'\otimes \eta)\\
=&\Sigma W^*(\mathrm{id}\otimes R)(\Lambda_{\varphi}\otimes \Lambda_{\varphi})(\Delta(R(\eta))(\eta'\otimes 1))\\
=&\Sigma W^* (\mathrm{id}\otimes R)(\Lambda_{\varphi}\otimes \Lambda_{\varphi})(\sum  R(\eta_{(2)})\eta'\otimes R(\eta_{(1)}))\\
=& \Sigma W^*(\Lambda_{\varphi}\otimes \Lambda_{\varphi})(\sum R(\eta_{(2)})\eta'\otimes \eta_{(1)})\\
=&(\Lambda_{\varphi}\otimes \Lambda_{\varphi}) (\sum \eta_{(2)}\otimes \eta_{(1)} R(\eta_{(3)})\eta').
\end{align*}

Therefore, we can see that for any $\eta\in \mathfrak{n}_{\varphi}\subseteq L^2(\g)$ we have
\[\widehat{W}(R\otimes \mathrm{id})\widehat{W}(R\otimes \mathrm{id})(|\eta\rangle \otimes 1)=\sum |\eta_{(2)}\rangle \otimes \eta_{(1)} R(\eta_{(3)}).\]

\begin{lemma}\label{prop2}
Let $\g$ be a co-amenable compact quantum group of Kac type or a discrete quantum group of Kac type. If $\eta\in \mathfrak{n}_{\varphi}\subseteq  L^2(\g)$ satisfies $\Delta(\eta)=(\sigma\circ \Delta)(\eta)=\sum \eta_{(2)}\otimes \eta_{(1)}$, then
\[\widehat{W}(R \otimes \mathrm{id})\widehat{W}(R \otimes \mathrm{id})(|\eta\rangle \otimes 1)=|\eta\rangle \otimes 1.\]

\end{lemma}
  
\begin{proof}
By the assumption and the co-associativity of $\Delta$, we have that
\[(\mathrm{id}\otimes \Delta)(\Delta(\eta))=(\Delta\otimes \mathrm{id})(\Delta(\eta))=\sum \eta_{(2)}\otimes \eta_{(3)}\otimes \eta_{(1)}.\]
Then
\begin{align*}
\sum \eta_{(2)}\otimes \eta_{(3)} R(\eta_{(1)})&=(\mathrm{id}\otimes m)(\mathrm{id}\otimes \mathrm{id}\otimes R)(\mathrm{id}\otimes \Delta)(\Delta(\eta))\\
&=(\mathrm{id}\otimes \epsilon')(\Delta(\eta))=\eta\otimes 1,
\end{align*}
where $m$ is the multiplication of $L^{\infty}(\g)$, $\epsilon'(a):=\epsilon(a)1$ and $\epsilon$ is the co-unit of $\g$. Then, by applying $\mathrm{id}\otimes R$ again, we have
\[\sum \eta_{(2)}\otimes \eta_{(1)} R(\eta_{(3)})=\eta\otimes 1.\]
\end{proof}

\begin{corollary}
Any compact quantum group of Kac type or amenable discrete quantum group of Kac type has the Day-Dixmier property with $d_{cb}(L^1(\g))\leq 2$.
\end{corollary}

\begin{proof}
In view of Lemma \ref{prop2} and Theorem \ref{thm1}, it is sufficient to show that, in both cases, there exists a net $(\xi_i)_i\subseteq \mathfrak{n}_{\widehat{\varphi}}\subseteq L^2(\widehat{\g})$ such that  $\left \|\xi_i\right\|_{L^2(\widehat{\g})}=1$, $\widehat{\Delta}(\xi_i)=\displaystyle \sum (\xi_i)_{(2)}\otimes (\xi_i)_{(1)}$ for all $i$ and $(w_{\xi_i,\xi_i})_i$ is a contractive approximate identity of $L^1(\widehat{\g})$.

For the case of $\g$ a compact quantum group of Kac type, the unit of $L^1(\widehat{\g})$ is given by $w_{E^0_{0,0},E^0_{0,0}}=\widehat{\varphi}(\cdot E^0_{0,0})$ when we write 
\[ L^{\infty}(\widehat{\g})=\displaystyle \ell^{\infty}-\bigoplus_{\alpha\in \mathrm{Irr}(\g)}M_{n_{\alpha}}=\ell^{\infty}-\bigoplus_{\alpha\in \mathrm{Irr}(\g)}\mathrm{span}\left \{E^{\alpha}_{i,j}:1\leq i,j\leq n_{\alpha}\right\} .\] Here, $\alpha=0$ means the trivial representation. Furthermore, we have
\[\widehat{\Delta}(E^0_{0,0})=\sum (E^0_{0,0})_{(1)}\otimes (E^0_{0,0})_{(2)}=\sum_{\alpha \in \mathrm{Irr}(\g)}\sum_{i,j=1}^{n_{\alpha}}\frac{1}{n_{\alpha}}E^{\alpha}_{i,j}\otimes E^{\overline{\alpha}}_{i,j}\]
where $\overline{\alpha}$ is the conjugate of $\alpha$. Hence $\widehat{\Delta}(E^0_{0,0})=\displaystyle \sum (E^0_{0,0})_{(2)}\otimes (E^0_{0,0})_{(1)}$.

Secondly, let us suppose that $\g$ is an amenable discrete quantum group of Kac type. It is equivalent to that $\widehat{\g}$ is a co-amenable compact quantum group of Kac type \cite{Ruan96},\cite{To06}. In this case, a contractive approximated identity of $L^1(\widehat{\g})$ is given as the form $w_{\xi_i,\xi_i}$ with $\xi_i\in \mathrm{span}\left \{\chi_{\alpha}:\alpha\in \mathrm{Irr}(\g)\right\}$ where $\chi_{\alpha}=\displaystyle \sum_{i=1}^{n_{\alpha}}u^{\alpha}_{i,i}$ for each $\alpha\in \mathrm{Irr}(\g)$ \cite{Br16},\cite{KrRu99}. Moreover, for $\xi_i=\displaystyle \sum_{\alpha\in \mathrm{Irr}(\g)}c^i_{\alpha}\chi_{\alpha}$, we have
\[\widehat{\Delta}(\xi_i)=\sum_{\alpha\in \mathrm{Irr}(\g)}c^i_{\alpha}\sum_{j,k=1}^{n_{\alpha}}u^{\alpha}_{j,k}\otimes u^{\alpha}_{k,j}.\]
Hence again we have $\widehat{\Delta}(\xi_i)=\displaystyle \sum (\xi_i)_{(2)}\otimes (\xi_i)_{(1)}$.
\end{proof}

\subsection{Example 2: Some Fourier algebras of crossed products}

Suppose that $N$ is discrete and $H$ is amenable. Also, we choose a net $(f_i)_i\subseteq L^2(H)_+$ such that \[\displaystyle \int_{H} \left |f_i(h_0h)-f_i(h)\right |^2dh\rightarrow 0 \quad \text{uniformly for $h_0$ on compact subsets of $H$}. \]

Now the operator $T:=\widehat{W}(\tau_2\circ (J_H\otimes J_N)\otimes \mathrm{id}\otimes \mathrm{id})\widehat{W}(\tau_2\circ (J_H\otimes J_N)\otimes \mathrm{id}\otimes \mathrm{id})\in B(L^2(H\times N\times H\times N))$ is computed by
\begin{align*}
&(\widehat{W}(\tau_2\circ (J_H\otimes J_N)\otimes \mathrm{id}\otimes \mathrm{id})\widehat{W}(\tau_2\circ (J_H\otimes J_N)\otimes \mathrm{id}\otimes \mathrm{id})f)(h_1,n_1,h_2,n_2)\\
&=f(h_1,\alpha_{h_1^{-1}h_2}(n_2)n_1\alpha_{h_1^{-1}h_2}(n_2^{-1}),h_2,n_2) 
\end{align*}
for any $f\in L^2(H\times N\times H\times N)$.

In particular,
\begin{align*}
\left |T^{i,j}\right|^2&:=\left |T (|f_i\rangle\otimes |\chi_{e_N}\rangle\otimes 1)-|f_i\rangle \otimes |\chi_{e_N}\rangle \otimes 1\right|^2=0~\mathrm{in~} B(L^2(H\times N))
\end{align*}
In other words, for any $\xi\in L^2(H\times N)$, $|f_i\rangle\otimes |\chi_{e_N}\rangle\otimes \xi$ is a fixed vector for the operator $T$, so that we get the similarity result.

\begin{corollary}
Suppose that $N$ is discrete and $H$ is amenable. Then the dual $\widehat{\g}$ of the crossed product $\g=(L^{\infty}(N)\rtimes_{\alpha} H,\Delta)$ has the Day-Dixmier property with $d_{cb}(A(\g))=d_{cb}(L^1(\widehat{\g}))\leq 2$.
\end{corollary}

\subsection{The case of completely bounded similarity degree 1}

One might wonder when $d_{cb}(L^1(\g))=1$ happens. As in \cite{Pi98}, \cite{Sp02}, \cite{LeSaSp16}, it is reasonable to conjecture that $d_{cb}(L^1(\g))=1$ if and only if the underlying quantum group $\g$ is finite, i.e. $L^{\infty}(\g)$ is finite dimensional.

Indeed, if $\g$ has the Day-Dixmier property with $d_{cb}(L^1(\g))=1$, then the map $m_1:L^1(\g)\rightarrow \widetilde{A}_1=C_0^u(\widehat{\g})$ becomes an isomorphism by [Theorem 4.2.9, \cite{Sp02}]. Note that $m_1$ is nothing but the universal Fourier transform $\lambda^u$ and $L^1(\g)$ has a bounded approximate identity.

Since $\g$ is co-amenable and $L^1(\g)$ is Arens regular, $\g$ is discrete by [Theorem 3.10, \cite{HuNeRu12}]. Then the surjectivity of $m_1=\lambda^u:L^1(\g)\rightarrow C_0^u(\widehat{\g})$ implies $\g$ is finite. In conclusion, a locally compact quantum group $\g$ has the Day-Dixmier property with $d_{cb}(L^1(\g))=1$ if and only if $\g$ is finite.

\bibliographystyle{alpha}
\bibliography{Dixmier.LCQG-revised}

\end{document}